\documentclass{amsart}

\usepackage{amssymb,latexsym,amsmath,amsthm,enumitem,mathrsfs,geometry,xcolor}
\usepackage{hyperref}

\newtheorem{thm}{Theorem}
\newtheorem*{thm*}{Theorem}
\newtheorem{cor}{Corollary}
\newtheorem{lem}{Lemma}
\newtheorem{prop}{Proposition}
\theoremstyle{definition}
\newtheorem{assume}{Assumption}

\theoremstyle{remark}
\newtheorem*{rmk*}{Remark}

\theoremstyle{plain}
\newtheorem*{FujiiThm}{Fujii's theorem}
\newtheorem*{ITSP}{The Error in PNT and Zero-Free Region (PNT-ZFR)}

\title[Average Goldbach and Zero-Free Regions]{The Average Number of Goldbach Representations and Zero-Free Regions of the Riemann Zeta-Function}

\author[Billington]{Keith Billington}
\address{Department of Mathematics and Statistics, San Jos\'{e} State University}
\email{keith.billington@sjsu.edu}

\author[Cheng]{Maddie Cheng}
\address{Department of Mathematics and Statistics, San Jos\'{e} State University}
\email{maddie.cheng@sjsu.edu}

\author[Schettler]{Jordan Schettler}
\address{Department of Mathematics and Statistics, San Jos\'{e} State University}
\email{jordan.schettler@sjsu.edu}

\author[Suriajaya]{Ade Irma Suriajaya}
\address{Faculty of Mathematics, Kyushu University}
\email{adeirmasuriajaya@math.kyushu-u.ac.jp}


\begin{document}

\maketitle

\begin{abstract}
In this paper, we prove an unconditional form of Fujii's formula for the average number of Goldbach representations and show that the error in this formula is determined by a general zero-free region of the Riemann zeta-function, and vice versa.
In particular, we describe the error in the unconditional formula in terms of the remainder in the Prime Number Theorem which connects the error to zero-free regions of the Riemann zeta-function.
\end{abstract}

\section{Introduction and Results}
\label{sec:intro}

We consider the function
\begin{equation*}%\label{psi_2}
\psi_2(n) := \sum_{m+m'=n} \Lambda(m)\Lambda(m'),
\end{equation*}
where $\Lambda$ is the von Mangoldt function, defined by $\Lambda(n)=\log p$ if $n=p^m$, $p$ a prime and $m\ge 1$, and $\Lambda(n)= 0$ otherwise. Thus $\psi_2(n)$ counts the number of \lq \lq Goldbach representations" of $n$ as sums of both primes and prime powers when weighted by the logarithm.
In this paper, we study the average number of Goldbach representations via the summatory function
$$ G(N) := \sum_{n\le N} \psi_2(n). $$
Here and throughout this paper all sums over integers start at 1 unless otherwise indicated.
Note, however, that $G(N)=0$ if $N<4$.
Fujii \cite{Fujii1,Fujii2,Fujii3} in 1991 obtained the following estimate for $G(N)$.
\begin{FujiiThm}[Fujii]\label{thm-Fujii}
Assuming the Riemann Hypothesis, we have for all sufficiently large $N$ that
\begin{equation}\label{FujiiThm}
G(N) = \frac{N^2}{2} - 2\sum_{\rho} \frac{N^{\rho+1}}{\rho(\rho+1)} + O(N^{\frac43}(\log N)^{\frac43}),
\end{equation}
where the sum is over zeros $\rho=\beta+i\gamma$ of the Riemann zeta-function $\zeta(s)$ satisfying $\gamma\neq0$, and the Riemann Hypothesis is $\beta = 1/2$. This sum over zeros is absolutely convergent. 
\end{FujiiThm}

The error term in Fujii's formula \eqref{FujiiThm} above was improved by Bhowmik and Schlage-Puchta in \cite{BhowPuchta2010} to $O(N\log^5\!N)$.
This was further refined by Languasco and Zaccagnini \cite{Lang-Zac1} to $O(N\log^3\!N)$, also reproved by Goldston and Yang \cite{Gold-Yang}.
The oscillation of the second term in \eqref{FujiiThm} (the sum over zeros) has been studied by Mossinghoff and Trudgian \cite{MT22}, building upon the work of Fujii \cite{Fujii3}.
These are all obtained assuming the Riemann Hypothesis (RH).

Not limited to improvements for the error, this Fujii-type formula \eqref{FujiiThm} for $G(N)$ has been studied by many authors in various contexts. This includes the pioneering work of Granville \cite{Gran1,Gran2}, followed by important work of Bhowmik and Ruzsa \cite{BhowRuzsa2018}, Bhowmik, Halupczok, Matsumoto, and Suzuki \cite{Bhowmik-H-M-SGoldbach2019}, and also Goldston and the fourth author \cite{GoldSur-Fujii,GoldSur-forWirsing}.
It is known that the error in \eqref{FujiiThm} is related to estimates for primes in short intervals and the method used to obtain \hyperref[thm-Fujii]{Fujii's theorem} also gives an estimate for the error in the Prime Number Theorem (PNT). The latter was further studied by Goldston and the fourth author in their other work \cite{GoldSur-PNT-PC}.
\medskip

Classical proofs of PNT have revealed how a zero-free region of the Riemann zeta-function affects the error estimate in PNT. The details can be found in many books such as \cite[Chapter III Sections 10 to 12]{Ingham1932}, \cite[Chapter 6, Chapter 12 Section 12.1, and Chapter 13 Section 13.1]{MontgomeryVaughan2007} and \cite[Chapter III]{Titchmarsh}.
Conversely, Tur\'an \cite{Turan1950,TuranBook}, Sta\'s \cite{Sta61}, and Pintz \cite{Pintz1980,Pintz1984} showed that an error estimate for PNT implies a zero-free region for the Riemann zeta-function. 
We state these results in a complete form as follows.
We denote the remainder term in PNT by
\begin{equation}\label{PNT-remainder}
R(x):= \psi(x) - x,
\end{equation}
where $\psi(x) := \sum_{n\le x} \Lambda(n)$ counts the number of primes and prime powers up to $x$ weighted by the logarithm.
\begin{ITSP}[Ingham, Tur\'an, Sta\'s, and Pintz] \label{thm-InghamPintz}
Let $0<\varepsilon<1$ be fixed and $\eta(u)$ be a continuous decreasing function of $u\ge0$ satisfying $0<\eta(u)\le1/2$.
\begin{enumerate}[label={(\alph*)}]
\item\label{Ingham-a} Suppose that $\zeta(\sigma+it)\neq0$ in the region
$$ \sigma > 1 - \eta(|t|), $$
then we have
$$ R(x) = O(x\exp(-(1-\varepsilon)\omega(x))), $$
where \begin{equation}\label{Ingham_omega2}
\omega(x) := \min_{u\ge1}(\eta(u)\log{x}+\log{u}).
\end{equation}

\item\label{Ingham-b} Conversely, if 
$$ R(x) = O(x\exp(-(1+\varepsilon)\varpi(x))) $$
where
\begin{equation}\label{converse_omega}
\varpi(x) := \min_{u\ge0}(\eta(u)\log{x}+u),
\end{equation}
and in addition, $\eta(u)$ satisfies $\lim_{u\to\infty}\eta'(u)=0$, then for all sufficiently large $|t|$, $\zeta(\sigma+it)\neq0$ in the region
$$ \sigma > 1 - \eta(\log|t|). $$
\end{enumerate}
\end{ITSP}

In this paper, we establish the Goldbach-analogue of the above \hyperref[thm-InghamPintz]{PNT-ZFR}.
We state these results as our first two theorems.
The first theorem is the analogue of \ref{Ingham-a} of \hyperref[thm-InghamPintz]{PNT-ZFR}.
\begin{thm}\label{thm1}
If $\zeta(\sigma+it)\neq0$ in the region
$$ \sigma > 1 - \eta(|t|), $$
for a function $\eta(u)$ satisfying the conditions in \hyperref[thm-InghamPintz]{PNT-ZFR}, then there exists a constant $0<C<1$ such that for $N\ge 4$
\begin{equation}\label{G(N)-asymp}
G(N) = \frac{N^2}{2} + O\left(N^2\exp(-C\omega(N))\right).
\end{equation}
where $\omega(x)$ is as defined in \eqref{Ingham_omega2}.
\end{thm}

From the above theorem, we immediately obtain the following corollary, assuming an explicit form of a zero-free region of the Riemann zeta-function.
\begin{cor}\label{cor1-1}
Suppose that $\zeta(\sigma+it)\neq0$ in the region
$$ \sigma > 1 - \eta(|t|), $$
where
\begin{equation*}%\label{KorVino-shape-zfr}
\eta(u) = \frac{c}{(\log{(u+3)})^a(\log\log{(u+3)})^b}
\end{equation*}
for some $c>0$, $0\le a\le1$, and $b\in\mathbb{R}$ but $b\ge0$ if $a=0$.
Then for any $N\ge4$, there exists a constant $c'>0$ such that
\[ G(N) = \frac{N^2}{2} + O\left(N^2 \exp\left(-c'(\log{N})^{\frac1{1+a}}(\log\log N)^{-\frac{b}{1+a}}\right)\right). \]
\end{cor}
\noindent
We note that the currently best known zero-free region is the case when $a=2/3$ and $b=1/3$, due to Korobov \cite{koro,koro2} and Vinogradov \cite{vino} independently.
\medskip

To obtain a converse of Theorem \ref{thm1}, we need to bound either $R(x)$ or, correspondingly, the function
\begin{equation}\label{R1}
R_1(x) := \psi_1(x) - \frac{x^2}{2} = \int_0^x R(u)\, du,
\end{equation}
where
\begin{equation}\label{psi1}
\psi_1(x) := \int_0^x \psi(u)\,du = \sum_{n\le x}(x-n)\Lambda(n)
\end{equation}
(see \cite[Equation (13) in Chapter II]{Ingham1932}), using a bound on
\begin{equation}\label{G-error}
E(N) := G(N) - \frac{N^2}{2}.
\end{equation}
Assuming an asymptotic for $G(N)$ of the form \eqref{G(N)-asymp}, Bhowmik and Ruzsa \cite{BhowRuzsa2018} have shown how to obtain a bound on $R(x)$ in the case when $\omega(x)=c\log{x}$, which leads to an equivalence with the quasi-Riemann Hypothesis (quasi-RH): $\zeta(s)\neq0$ in ${\rm Re}(s)>1-c$ for some constant $0<c<1/2$.
We generalize the method of \cite{BhowRuzsa2018} to a continuous increasing function $\varpi(x)$ with a few additional conditions.
Further, we deduce a bound on $R_1(x)$ instead of $R(x)$ which slightly improves the result.
Then applying \ref{Ingham-b} of \hyperref[thm-InghamPintz]{PNT-ZFR} (see also \cite[Theorem 2]{Pintz1980}), we obtain a zero-free region for the Riemann zeta-function as follows.

\begin{thm}\label{thm2}
Let $\varpi(x)$ be the function defined in \eqref{converse_omega} with $\eta(u)$ satisfying all the conditions of \hyperref[thm-InghamPintz]{PNT-ZFR} including the additional condition imposed in \ref{Ingham-b}.
Let $A$ be any fixed real number greater than $1$.
Then if for $N\ge4$ and some constant $C>5A$,
\begin{equation}\label{thm2-assump}
G(N) = \frac{N^2}{2} + O\left(N^2\exp(-C\varpi(N))\right),
\end{equation}
then for all sufficiently large $|t|$, $\zeta(\sigma+it)\neq0$ in the region
$$ \sigma > 1 - \eta(\log|t|). $$
\end{thm}
%We remark that Theorem \ref{thm2} does not give any converse theorem if the constant $C$ in \eqref{thm2-assump} is small. Or in other words, if we have a worse error bound than \eqref{thm2-assump}, Theorem \ref{thm2} does not apply. This is due to inefficiency of the proof and also the constraint coming from \ref{Ingham-b} of \hyperref[thm-InghamPintz]{PNT-ZFR} where the constant needs to be greater than $1$, i.e., $1+\varepsilon$ instead of $1-\varepsilon$.

As we will see in the next section, a zero-free region $\sigma > 1 - \eta(|t|)$ for the Riemann zeta-function satisfies either $\eta(|t|) \to 0$ as $|t|\to \infty$, or $\eta(u)$ being a constant function which refers to either the quasi-RH or RH case.
The result in Theorem \ref{thm2} is new and essentially optimal in the case when $\eta(|t|) \to 0$ as $|t|\to \infty$.
%We use the method of \cite{BhowRuzsa2018} to obtain Theorem \ref{thm2}.
%The result proved in \cite{BhowRuzsa2018} covered the quasi-RH case and also allowed $\eta(|t|)\to 0$ as $|t|\to 0$, but did not determine $\eta(u)$ further.
Our method, which is based on \cite{BhowRuzsa2018}, is inefficient in the quasi-RH and RH case. In these cases, however, an essentially best-possible result has been obtained by a different method in \cite[Theorem 1 (2)]{Bhowmik-H-M-SGoldbach2019}.
More precisely, they showed that if for $N\ge4$ and some constant $0<c\le1/2$, if
\[ G(N) = \frac{N^2}{2} + O\left(N^{2-c+\epsilon}\right) \]
holds for any $\epsilon>0$, then $\zeta(s)\neq0$ in the region ${\rm Re}(s)>1-c$.
\medskip

The method of Bhowmik and Ruzsa \cite{BhowRuzsa2018} uses a smooth power series average of $G(N)$, and most of the proof is concerned with removing the smooth weight in the final result. If instead we are content to use a power series generating function for our averaging, many of the difficulties in our proofs disappear, and we obtain the following result.

\begin{thm}\label{thm-smooth}
Let
\begin{equation*}%\label{smoothAvg}
F(N) := \sum_n \psi_2(n)e^{-n/N},
\end{equation*}
and suppose that 
$\zeta(\sigma+it)\neq0$ in the region
$$ \sigma > 1 - \eta(|t|), $$
for some continuous decreasing function $\eta(u)$ such that $0<\eta(u)\le1/2$.
Then for $N\ge3$, we have
\[ F(N) = N^2 + O\left(N^{2-\eta(\log{N})}\right). \]
\end{thm}

As a corollary to the above theorem, following Corollary \ref{cor1-1}, we also give an example using a well-known shape of a zero-free region of the Riemann zeta-function.
\begin{cor}\label{cor-smooth1}
Suppose that $\zeta(\sigma+it)\neq0$ in the region
$$ \sigma > 1 - \eta(|t|) $$
for $|t|\ge3$, where
\begin{equation*}%\label{KorVino-shape-zfr2}
\eta(u) = \frac{c}{(\log{u})^a(\log\log{u})^b}
\end{equation*}
for some $c>0$, $0\le a\le1$, and $b\in\mathbb{R}$ but $b\ge0$ if $a=0$.
Then for any $N\ge16$,
\[ F(N) = N^2 + O\left(N^{2}\exp\left(-\frac{c\log N}{{\left( \log{\log N} \right)^{a}}{\left( \log{\log{\log N}} \right)^{b}}}\right)\right). \]
\end{cor}

Similarly, we can also show a converse for Theorem \ref{thm-smooth}. We note, however, that this does not require the arguments in Section \ref{sec-converse} where we prove Theorem \ref{thm2}. Instead we make use of the explicit formula (see \eqref{PsiSmoothSumeq} in Section \ref{sec:smooth}) and apply the method of \cite{Pintz1980,Pintz1984}.

\section{Bounding the error term \texorpdfstring{$E(N)$}{E(N)} of \texorpdfstring{$G(N)$}{G(N)} --- Proof of Theorem \ref{thm1} and Corollary \ref{cor1-1}}
\label{sec:Goldbach-error}

The explicit formula
\begin{equation}\label{explicit-psi_1}
\psi_1(N) = \frac{N^2}2 - \sum_{\rho} \frac{N^{\rho+1}}{\rho(\rho+1)} + O(N),
\end{equation}
(see \cite[Chapter IV Theorem 28]{Ingham1932}) where the sum is as in \hyperref[thm-Fujii]{Fujii's theorem}, implies
\begin{equation}\label{explicit-R1}
R_1(N) = - \sum_{\rho} \frac{N^{\rho+1}}{\rho(\rho+1)} + O(N).
\end{equation}
Thus with this notation, Fujii's formula \eqref{FujiiThm} can be rewritten as
\[ G(N) = \frac{N^2}{2} + 2R_1(N) + O(N^{\frac43}(\log N)^{\frac43}), \]
and the currently best estimate under RH, due to Languasco and Zaccagnini \cite{Lang-Zac1}, is
\begin{equation}\label{LZ-RH_best}
G(N) = \frac{N^2}{2} + 2R_1(N) + O(N\log^3\!N).
\end{equation}
As a side note, assuming RH, we immediately obtain from \eqref{explicit-R1} that
\begin{equation}\label{R1_underRH}
R_1(N) \ll N^{3/2},
\end{equation}
which then implies
\begin{equation}\label{E_underRH}
G(N) = \frac{N^2}{2} + O(N^{3/2}).
\end{equation}
We also remark that Bhowmik and Schlage-Puchta in \cite{BhowPuchta2010} also showed unconditionally that 
\begin{equation*}%\label{uncond-omega}
G(N) = \frac{N^2}{2} + 2R_1(N) + \Omega(N\log\log{N}),
\end{equation*}
which suggests that the estimate \eqref{LZ-RH_best} is close to being sharp.

By \eqref{explicit-R1}, using any zero of the Riemann zeta-function on the line ${\rm Re}(s)=1/2$, such as $\rho=14.134725\ldots$, we immediately have
\begin{equation}\label{R1-Omega}
R_1(N) = \Omega_\pm(N^{3/2}).
\end{equation}
This can be done using classical methods given in, for example, \cite[Theorem 33]{Ingham1932} or \cite[Theorem 15.3]{MontgomeryVaughan2007}.
Denoting by $E(N)$ the error for $G(N)$, as defined in \eqref{G-error}, the estimate \eqref{R1-Omega} and the discussion in the previous paragraph suggest that $R_1(N)$ is likely the main term of $E(N)$.
It is, however, difficult to draw such a conclusion, unless we can detect the interaction between the oscillation of the term
$$ G(N) - \frac{N^2}{2} - 2R_1(N) $$
and the oscillation of $R_1(N)$ itself. The following lemma clarifies this statement.

\begin{lem}\label{lem1}
Define $\Lambda_0(n) = \Lambda(n)-1$ for $n\ge 1$ and $\Lambda_0(n)=0$ otherwise. Then for $N\ge4$, we have 
\begin{equation}\label{lem1:2}
G(N) = \frac{N^2}{2} + 2R_1(N) + \sum_{n\le N}\Lambda_0(n)R(N-n) + O(N).
\end{equation}
\end{lem}
\begin{proof}[Proof of Lemma \ref{lem1}] Letting 
\[ \psi^0_2(n) := \sum_{m+m'=n}\Lambda_0(m)\Lambda_0(m'), \]
we have 
\begin{equation}\label{psi0_2-eq}
\begin{split} \psi^0_2(n) &= \sum_{m+m'=n}\big( \Lambda(m)\Lambda(m') -\Lambda(m) -\Lambda(m') +1\big) \\& = \psi_2(n) - 2\sum_{m+m'=n}\Lambda(m) + \sum_{m+m'=n}1 \\&
= \psi_2(n) - 2\psi(n-1) + (n-1).
\end{split}
\end{equation}
By \eqref{psi1} we have $\psi_1(N)=\sum_{n\le N}\psi(n-1)$, hence
\[ \begin{split} G(N) &= \sum_{n\le N} \psi_2(n) = \sum_{n\le N}\big( 2\psi(n-1) - (n-1) \big) + \sum_{n\le N} \psi^0_2(n) \\&
= 2\psi_1(N) -\frac{(N-1)N}{2} + \sum_{n\le N} \psi^0_2(n).
\end{split} \]
Next, equations \eqref{explicit-psi_1} and \eqref{explicit-R1} immediately give
\begin{equation}\label{lem1'}
G(N) = \frac{N^2}{2} + 2R_1(N) + \sum_{n\le N} \psi^0_2(n) + O(N).
\end{equation}
Finally, by \eqref{PNT-remainder},
\begin{equation*}
\begin{aligned}
\sum_{n\le N} \psi^0_2(n)
%&= \sum_{m+m'\le N}\Lambda_0(m)\Lambda_0(m')
= \sum_{m\le N} \Lambda_0(m)\left(\sum_{m'\le N-m} \Lambda_0(m')\right)
&= \sum_{m\le N} \Lambda_0(m)(\psi(N-m)-(N-m)) \\
&= \sum_{m\le N} \Lambda_0(m)R(N-m),
\end{aligned}
\end{equation*}
which gives \eqref{lem1:2}.
\end{proof}

The explicit formula for $\psi(x)$ as stated in \cite[Equation (1.7)]{Gold83}, originally due to \cite{Landau1908}, gives for $x, T\ge 3$, 
\begin{equation}\label{explicit-psi}
R(x) = - \sum_{\substack{\rho \\ |\gamma|\le T}} \frac{x^\rho}{\rho} + O\left(\frac{x\log x\log\log x}{T}\right) + O(\log x),
\end{equation}
where the sum is over zeros $\rho=\beta +i\gamma$ of the Riemann zeta-function satisfying $0<|\gamma|\le T$.
This shows that $R(x)$ is an oscillating function of $x$ and is determined by the distribution of zeros of the Riemann zeta-function, as is $R_1(x)$ which is clear from the explicit formula \eqref{explicit-R1}.
The equation \eqref{lem1:2} in Lemma \ref{lem1} indicates that by separating the error into an $R_1(N)$ term and an average over $R(x)$, we cannot rule out the possibility of the two oscillating terms canceling each other. This also results in difficulty connecting the bound on $E(N)$ directly to zero-free regions of the Riemann zeta-function.

We describe a zero-free region of the Riemann zeta-function as the region
\begin{equation*}%\label{zfr}
\sigma > 1 - \eta(|t|)
\end{equation*}
such that $\zeta(\sigma+it)\neq0$.
Here $\eta(u)$ is a continuous decreasing function satisfying $0<\eta(u)\le1/2$.
If we set
\[ \Theta := \sup\{\beta : \rho=\beta+i\gamma,\, \zeta(\rho)=0\}, \]
we have that $1/2\le\Theta\le1$. If $1/2\le\Theta<1$, then we can write $\eta(u)$ as a constant function $\eta(u)=\Theta$.
The case $1/2<\Theta<1$ is the so-called quasi-RH, and RH is when $\Theta=1/2$.
If $\Theta=1$, then $\eta(|t|)\to0$ as $|t|\to \infty$.
We further remark that the classical zero-free region of the Riemann zeta-function is the case when
\begin{equation}\label{classicZFR}
\eta(u) = \frac{c}{\log(u+3)}
\end{equation}
for some constant $c>0$, and the current best zero-free region known is the case
\begin{equation}\label{KorVinoZFR}
\eta(u) = \frac{c}{(\log{(u+3)})^{2/3}(\log\log{(u+3)})^{1/3}}, 
\end{equation}
for some constant $c>0$, which is due to Korobov \cite{koro,koro2} and Vinogradov \cite{vino} independently, see also \cite[p. 176]{Mont94}.

As seen in \hyperref[thm-InghamPintz]{PNT-ZFR}, by the explicit formula \eqref{explicit-psi} and the shape of the zero-free region of the Riemann zeta-function, $R(x)$ is bounded above by a nonnegative function of the form
\begin{equation}\label{Rbound}
x\exp(-C\omega(x))
\end{equation}
for some constant $C>0$, where $\omega(x)$ is as defined in \eqref{Ingham_omega2}.
Correspondingly, we see from the explicit formula \eqref{explicit-R1} that $R_1(x)$ is bounded by a function of the form $x^2\exp(-C\omega(x))$.
This then determines the bound for $E(N)$.
In the following proposition, we utilize Lemma \ref{lem1} to obtain an upper bound of $E(N)$ using a bound on either $R(x)$ or $R_1(x)$.
\begin{prop}\label{thm1-interm}
The average number of Goldbach representations is
\begin{equation}\label{thm1-1}
G(N) = \frac{N^2}{2} + O\left(N\max_{0\le u\le N}\left|R(u)\right| \right).
\end{equation}
Alternatively, this can be expressed in terms of $R_1(x)$ as
\begin{equation}\label{thm1-2}
G(N) = \frac{N^2}{2} + O\left(|R_1(N)| + N\sqrt{\max_{0\le u\le \frac{3N}2} |R_1(u)|} \right).
\end{equation}
\end{prop}

We remark that the bounds \eqref{thm1-1} and \eqref{thm1-2} are only essentially equivalent in the zero-free region case where $\eta(|t|)\to0$ as $|t|\to \infty$.
In the RH case, \eqref{thm1-1} gives about the expected correct error bound, while \eqref{thm1-2} degrades the error.
We explain this here briefly.
von Koch \cite{vonKoch1901} showed that RH implies
$$ R(x) = O(x^{1/2}\log^2\!x), $$
meanwhile it has been shown by Schmidt \cite{Sch03} that
\begin{equation}\label{R-omega}
R(x) = \Omega_\pm\!(x^{1/2}).
\end{equation}
%thus except for a factor of $\log^2\!x$, the best possible bound for $R(x)$ is given by RH.
Using Lemma \ref{lem1} along with PNT and the $\Omega$-estimates \eqref{R1-Omega} and \eqref{R-omega}, we see that
\begin{equation*}%\label{G-Omega}
G(N) = \frac{N^2}{2} + \Omega_\pm(N^{3/2}).
\end{equation*}
We thus recall from \eqref{E_underRH}, that RH gives the best possible bound for $E(N)$.
By \eqref{thm1-1} of Proposition \ref{thm1-interm}, assuming RH we have
$$ G(N) = \frac{N^2}{2} + O(N^{3/2}\log^2\!N), $$
which is close to the right RH-bound \eqref{E_underRH} except for a $\log$ factor.
In view of \eqref{thm1-2} of Proposition \ref{thm1-interm}, even though the estimate \eqref{R1_underRH} gives the right bound, the second term in the error of \eqref{thm1-2} introduces an extra $1/4$ in the exponent.
A similar situation occurs as well in the quasi-RH case.

\begin{proof}[Proof of Proposition \ref{thm1-interm}]
By Lemma \ref{lem1}, it is immediate that
\begin{equation}\label{GwithRs}
\begin{aligned}
G(N) - \frac{N^2}{2} &\ll |R_1(N)| + \left|\sum_{n\le N}\Lambda_0(n)R(N-n)\right| + N \\
&\le |R_1(N)| + \max_{0\le u\le N}\left|R(u)\right| \sum_{n\le N} \left|\Lambda(n)-1\right| + N \\
&\le |R_1(N)| + \max_{0\le u\le N}\left|R(u)\right| (R(N) + 2N) + N \\
&\ll |R_1(N)| + N\max_{0\le u\le N}\left|R(u)\right|,
\end{aligned}
\end{equation}
taking into account \eqref{R-omega}.
Recalling \eqref{R1}, we can trivially bound $R_1$ as
\begin{equation}\label{R1byR}
R_1(x) = \int_0^x R(u)\, du \le x\max_{0\le u\le x}|R(u)|.
\end{equation}
Hence
\begin{align*}%\label{GwithR}
G(N) - \frac{N^2}{2} \ll N\max_{0\le u\le N}\left|R(u)\right|,
\end{align*}
which is the first estimate in Proposition \ref{thm1-interm}.

To prove the second estimate, we follow the argument of Ingham \cite[p. 64]{Ingham1932}.
Since $\psi(x)$ is an increasing function of $x\ge0$, we have for any $x>0$ and $0<h<x/2$,
\begin{equation}\label{psi-ineq}
\frac1h \int_{x-h}^x \psi(u)\,du \le \psi(x) \le \frac1h \int_x^{x+h} \psi(u)\,du.
\end{equation}
Recalling \eqref{psi1} and again \eqref{R1}, we see that the integrals on both sides of \eqref{psi-ineq} are
\begin{align*}
\frac{\psi_1(x\pm h) - \psi_1(x)}{\pm h} = x \pm \frac{h}2 \pm \frac{R_1(x\pm h) - R_1(x)}h.
\end{align*}
Thus
\begin{align*}
\frac{\psi_1(x\pm h) - \psi_1(x)}{\pm h} - x
\ll h + \frac1h \max_{\frac{x}2\le u \le \frac{3x}2} |R_1(u)|.
\end{align*}
Balancing the two terms on the right-hand side above shows that taking
$$ h = \frac{\sqrt{\displaystyle \max_{\frac{x}2\le u \le \frac{3x}2} |R_1(u)|}}2 $$
gives the right size of the bound and ensures $h<x/2$.
Thus
\begin{align*}
\frac{\psi_1(x\pm h) - \psi_1(x)}{\pm h} - x
\ll \sqrt{\max_{\frac{x}2\le u \le \frac{3x}2} |R_1(u)|}
\end{align*}
and by the inequalities \eqref{psi-ineq}, we obtain
$$ \psi(x) = x + O\left(\sqrt{\max_{\frac{x}2\le u \le \frac{3x}2} |R_1(u)|}\right), $$
which implies
\begin{equation}\label{RfromR1}
R(x) \ll \sqrt{\max_{\frac{x}2\le u \le \frac{3x}2} |R_1(u)|}.
\end{equation}
The estimate \eqref{GwithRs} then implies
$$ G(N) - \frac{N^2}{2} \ll |R_1(N)| + N\sqrt{\max_{0\le u \le \frac{3N}2} |R_1(u)|}. $$
\end{proof}

We next show a modified version of Proposition \ref{thm1-interm} in terms of bounds of the shape \eqref{Rbound}.
\begin{prop}\label{cor4}
The error in the formula for the average number of Goldbach representations is determined by the error in PNT as follows.
\begin{enumerate}[label={(\alph*)}]
\item\label{cor4a}
Assume that for any sufficiently large $x>0$,
$$ R(x) \ll x\exp(-Cf(x)) $$
for some constant $C>0$ and some continuous increasing function $f(x)$ of $x\ge0$, then
$$ R_1(x) \ll x^2\exp(-Cf(x)) $$
for any sufficiently large $x>0$, and for $N\ge4$,
\[ G(N) = \frac{N^2}{2} + O\left(N^2\exp(-Cf(N))\right). \]

\item\label{cor4b}
Similarly, suppose that for any sufficiently large $x>0$,
$$ R_1(x) \ll x^2\exp(-C_1g(x)) $$
for some constant $C_1>0$ and some continuous increasing function $g(x)$ of $x\ge0$, then for any sufficiently large $x>0$ and $N\ge4$, we have
$$ R(x) \ll x\exp\left(-\frac{C_1}2g\left(\frac{x}2\right)\right) $$
and
\[ G(N) = \frac{N^2}{2} + O\left(N^2\exp\left(-\frac{C_1}2g\left( \frac{N}2\right)\right)\right). \]
\end{enumerate}
\end{prop}

\begin{proof}[Proof of Proposition \ref{cor4}]
Assume that for some constant $C>0$ and some continuous increasing function $f(x)$ of $x\ge0$,
$$ R(x) \ll x\exp(-Cf(x)) $$
for any $x>0$.
Then \eqref{R1byR} implies that for any $x>0$,
$$ R_1(x) \ll x^2\exp(-Cf(x)) $$
and by \eqref{thm1-1} of Proposition \ref{thm1-interm}, we have for any $N\ge4$,
\begin{align*}
G(N) = \frac{N^2}{2} + O\left( N^2\exp(-Cf(N))\right).
\end{align*}
This proves \ref{cor4a} of Proposition \ref{cor4}.

Now suppose that for $x>0$, we have
$$ R_1(x) \ll x^2\exp(-C_1g(x)) $$
for some constant $C_1>0$ and some continuous increasing function $g(x)$ of $x\ge0$.
Then \eqref{RfromR1} implies that for any sufficiently large $x>0$,
$$ R(x) \ll x\exp\left(-\frac{C_1}2g\left(\frac{x}2\right)\right), $$
and by \eqref{thm1-2} of Proposition \ref{thm1-interm}, we have
\[ G(N) = \frac{N^2}{2} + O\left(N^2\exp\left(-\frac{C_1}2g\left(\frac{N}2\right)\right)\right) \]
for any $N\ge4$.
This is \ref{cor4b} of Proposition \ref{cor4}.
\end{proof}

Theorem \ref{thm1} is an immediate consequence of the above Proposition \ref{cor4} and \ref{Ingham-a} of \hyperref[thm-InghamPintz]{PNT-ZFR}.
\begin{proof}[Proof of Theorem \ref{thm1}]
Let $\eta(u)$ be a function satisfying the conditions in \hyperref[thm-InghamPintz]{PNT-ZFR}.
Suppose that $\zeta(\sigma+it)\neq0$ in the region
$$ \sigma > 1 - \eta(|t|). $$
Applying \ref{Ingham-a} of \hyperref[thm-InghamPintz]{PNT-ZFR}, we have for some constant $0<C<1$ that
\[ R(x) = O(x\exp(-C\omega(x))), \]
where $\omega(x)$ is as defined in \eqref{Ingham_omega2}.
Then \ref{cor4a} of Proposition \ref{cor4} implies that
\[ G(N) = \frac{N^2}{2} + O\left(N^2\exp(-C\omega(N))\right) \]
holds for any $N\ge4$.
\end{proof}

\begin{proof}[Proof of Corollary \ref{cor1-1}]
Applying Theorem \ref{thm1}, the assumption with $\eta(u)$ being
\begin{equation}\label{zfr-assump}
\eta(u) = \frac{c}{(\log{(u+3)})^a(\log\log{(u+3)})^b},
\quad c>0, \text{ and } (a,b)\in[0,1]\times\mathbb{R} \text{ or } (a,b)\in\{0\}\times\mathbb{R}_{\ge0},
\end{equation}
implies for any $N\ge4$,
\[ G(N) = \frac{N^2}{2} + O\left(N^2\exp(-C\omega(N))\right), \]
where $\omega(x)$ is as defined in \eqref{Ingham_omega2}.
%\[ \omega(x) = \min_{u\ge1}(\eta(u)\log{x}+\log{u}). \]
Since $\eta(u)\log{x}$ is a positive decreasing function of $u\ge1$ and $\log{u}$ is a positive increasing function of $u\ge1$, the minimum of $\eta(u)\log{x}+\log{u}$ as a function of $u\ge1$ is attained at the critical point $u=u_0$ which is determined by
\begin{equation}\label{min_pt}
u_0\eta'(u_0) = -\frac1{\log{x}}.
\end{equation}
Solving \eqref{min_pt} for $u_0$ with $\eta(u)$ as in \eqref{zfr-assump}, we have
$$
\log\log{u_0} = \frac{\log\log{x}}{1+a} \left( 1+ O\left(\frac{\log\log\log{x}}{\log\log{x}}\right) \right)
$$
by taking logarithm on both sides of \eqref{min_pt}.
Substituting this back into \eqref{min_pt} we obtain
$$
\log{u_0} = \left(\frac{\mathfrak{a}c(1+a)^b\log{x}}{(\log\log{x})^b}\right)^{\frac1{1+a}} \left( 1+ O\left(\frac{\log\log\log{x}}{\log\log{x}}\right) \right),
\quad \text{where } \mathfrak{a} :=
\begin{cases}
1 &\text{ if } a=0, \\
a &\text{ if } 0<a\le1.
\end{cases}
$$
Hence
$$
\omega(x) = (1+a)\left(\frac{\mathfrak{a}c(1+a)^b\log{x}}{(\log\log{x})^b}\right)^{\frac1{1+a}} \left( 1+ O\left(\frac{\log\log\log{x}}{\log\log{x}}\right) \right)
$$
which gives
\begin{align*}
G(N) = \frac{N^2}{2} + O\left(N^2\exp\left(-((1+a)C-\varepsilon)(\mathfrak{a}c(1+a)^b)^{\frac1{1+a}}(\log{N})^{\frac1{1+a}}(\log\log N)^{-\frac{b}{1+a}}\right)\right),
%\left(1 + O\left(\frac{\log\log\log{N}}{\log\log{N}}\right)\right)
\end{align*}
for any small $0<\varepsilon<1$.
Thus there exists a constant $c'>0$ such that
\[ G(N) = \frac{N^2}{2} + O\left(N^2 \exp\left(-c'(\log{N})^{\frac1{1+a}}(\log\log N)^{-\frac{b}{1+a}}\right)\right). \]
This completes the proof of Corollary \ref{cor1-1} and concludes this section.
\end{proof}

\section{The error for \texorpdfstring{$G(N)$}{G(N)} and the remainder term \texorpdfstring{$R(x)$}{R(x)} in PNT --- Proof of Theorem \ref{thm2}}
\label{sec-converse}

In this section, we establish a converse of Theorem \ref{thm1}.
Suppose that for some continuous increasing positive function $f$,
\[ G(N) = \frac{N^2}{2} + O\left(N^2\exp(-f(N))\right) \]
for $N\ge4$.
We first note that it suffices to assume that
\begin{equation*}
f(N) \le \frac12\log{N},
\end{equation*}
due to the omega-bound \eqref{R1-Omega}.
On the other hand, by Theorem \ref{thm1}, the classical zero-free region \eqref{classicZFR} implies that
\begin{equation*}
f(N) \gg \sqrt{\log{N}},
\end{equation*}
or even better, using the Korobov-Vinogradov zero-free region \eqref{KorVinoZFR} it suffices to assume
\begin{equation*}
f(N) \gg (\log{N})^{3/5}(\log\log{N})^{1/5}.
\end{equation*}
We now make the following assumption on the error term in the asymptotic formula for $G(N)$.
\begin{assume}\label{Gassumption}
For any $N\ge4$, we have
\[ G(N) = \frac{N^2}{2} + O\left(N^2\exp(-f(N))\right), \]
where $f$ is a continuous increasing function satisfying
\begin{equation*}
(\log{x})^{3/5}(\log\log{x})^{1/5} \ll f(x) \le \frac12\log{x}, \qquad \text{for all }~ x\ge3.
\end{equation*}
\end{assume}

The goal here is to prove the following theorem and use it to deduce Theorem \ref{thm2}.
\begin{prop}\label{thm2-BRgen}
Suppose $G(N)$ is as in Assumption \ref{Gassumption}. Then there exists a constant $C>0$ such that
\[ R(x) = O\left(x\exp\left(-Cf\left(\left(\frac{x}2\right)^{1/A}\right)\right)\right) \]
holds for $x\ge3$ and any real $A>1$.
\end{prop}
To prove Proposition \ref{thm2-BRgen}, we follow the argument of Bhowmik and Ruzsa \cite{BhowRuzsa2018} where we use a smoothed version of the average $G(N)$.
Before we begin, we remark that it follows from the proof that the constant $C$ in Proposition \ref{thm2-BRgen} satisfies $0<C<1$.
Let
\begin{equation}\label{smoothedPsi}
\Psi(z) := \sum_n\Lambda(n)z^n,
\end{equation}
where the sum runs over all positive integers $n$.
Using the above notation, we can then write
\begin{equation}\label{smoothedPsisquare}
\begin{aligned}
\Psi(z)^2 = \sum_m\sum_{m'}\Lambda(m)\Lambda(m')z^{m+m'}
= \sum_n \psi_2(n)z^n,
\end{aligned}
\end{equation}
which is a smoothed version of $G(N)$.

\begin{lem}\label{lem2-thm2}
Suppose $G(N)$ is as in Assumption \ref{Gassumption} and let $|z|=e^{-1/N}$ where $N\ge4$.
For any real $A>1$, we have
\begin{equation*}%\label{lem2-thm2-eq}
\Psi(z)^2 = \frac1{(1-z)^2} + O\left(|1-z|N^3\exp\left(-f\left(N^{1/A}\right)\right)\right).
\end{equation*}
\end{lem}

\begin{proof}[Proof of Lemma \ref{lem2-thm2}] We apply \eqref{smoothedPsisquare} and rewrite $\psi_2(n)$ as $G(n)-G(n-1)$ so that
\begin{align*}
\Psi(z)^2 = \sum_n \psi_2(n)z^n
= \sum_n\left(G(n)-G(n-1)\right)z^n
= (1-z)\sum_nG(n)z^n.
\end{align*}
By Assumption \ref{Gassumption},
\begin{align*}
\Psi(z)^2 &= (1-z)\sum_n\left(\frac{n^2}{2} + O\left(n^2\exp(-f(n))\right)\right)z^n \\
&= (1-z)\left(\sum_n\frac{n^2}2z^n + O\left(\left|\sum_nn^2z^n\exp(-f(n))\right|\right)\right).
\end{align*}
Next using the power series expansion
$$
\sum_nn^2z^n = \frac2{(1-z)^3}-\frac3{(1-z)^2} + \frac1{1-z} \qquad \text{for} \quad |z|<1,
$$
for the main term, we have
\begin{align} \label{psi2-expansion}
\Psi(z)^2 = \frac1{(1-z)^2}-\frac{3/2}{1-z} + \frac12 + O\left(|1-z|\sum_nn^2|z|^n\exp(-f(n))\right).
\end{align}
The error term in \eqref{psi2-expansion} is handled as follows. Since $f(x)$ is continuously increasing and satisfies the lower bound in Assumption \ref{Gassumption}, for any real $A>1$, we have for $|z|=e^{-1/N}$, 
\begin{align*}
\sum_n n^2|z|^n\exp(-f(n))
&= \left(\sum_{n\le N^{1/A}} + \sum_{N^{1/A}<n\le N} + \sum_{n>N}\right) n^2|z|^n\exp(-f(n)) \\
&\ll N^{2/A} \sum_{n\le N^{1/A}} 1 + N^3\exp\left(-f\left(N^{1/A}\right)\right) + N^3\exp\left(-f(N)\right) \\
&\ll N^{3/A} + N^3\exp\left(-f\left(N^{1/A}\right)\right) \\
&\ll N^3\exp\left(-f\left(N^{1/A}\right)\right).
\end{align*}
It is clear that the third term in \eqref{psi2-expansion}, which is only a constant, can be absorbed in the above error, while for the second term, we can use for example, the power series expansion of the exponential function which gives $(1-|z|)^{-1} \sim N$. Thus \eqref{psi2-expansion} becomes
\begin{align*}
\Psi(z)^2
&= \frac1{(1-z)^2} + O(N) + O\left(|1-z|N^3\exp\left(-f\left(N^{1/A}\right)\right)\right) \\
&= \frac1{(1-z)^2} + O\left(|1-z|N^3\exp\left(-f\left(N^{1/A}\right)\right)\right).
\end{align*}
\end{proof}

We next remove the square in Lemma \ref{lem2-thm2} to obtain an asymptotic formula for $\Psi(z)$.
\begin{lem}\label{lem3-thm2}
Let $N\ge4$. Suppose $G(N)$ is as in Assumption \ref{Gassumption} and $|z|=e^{-1/N}$.
For any real $A>1$, we have
\[ \Psi(z) = \frac1{1-z} + O\left(|1-z|^2N^3\exp\left(-f\left(N^{1/A}\right)\right)\right). \]
\end{lem}

\begin{proof}[Proof of Lemma \ref{lem3-thm2}]
Multiplying both sides of the equation in Lemma \ref{lem2-thm2} by $(1-z)^2$, we have
\begin{align*}
(1-z)^2\Psi(z)^2 = 1 + O\left(|1-z|^3N^3\exp\left(-f\left(N^{1/A}\right)\right)\right).
\end{align*}
Note that $\Psi(x)>0$ and $(1-x)>0$ for $0<x<1$, thus taking square roots we have
\begin{align*}
(1-z)\Psi(z) = \left(1 + O\left(|1-z|^3N^3\exp\left(-f\left(N^{1/A}\right)\right)\right)\right)^{1/2}.
\end{align*}
For convenience, we write the equation above as
\[ (1-z)\Psi(z) = \left(1 + \mathcal{E}(z)\right)^{1/2}. \]

If $|\mathcal{E}(z)|<1$, we can use a Taylor expansion on the right-hand side
to show that it is $1 + O\left(|\mathcal{E}(z)|\right)$, hence
\begin{align*}
(1-z)\Psi(z) = 1 + O\left(|1-z|^3N^3\exp\left(-f\left(N^{1/A}\right)\right)\right)
\end{align*}
which implies the lemma in this case.
Now if $|\mathcal{E}(z)|\ge1$, then
\begin{align*}
(1-z)\Psi(z) = \left(1 + \mathcal{E}(z)\right)^{1/2}
\ll |\mathcal{E}(z)|^{1/2} \ll |\mathcal{E}(z)|.
\end{align*}
Thus
\begin{align*}
\Psi(z) \ll \frac{|\mathcal{E}(z)|}{|1-z|}.
\end{align*}
Since $|\mathcal{E}(x)|\ge1$, on the right-hand side we can add $|1-z|^{-1}$ so that
\begin{align*}
\Psi(z) \ll \frac{|\mathcal{E}(z)|}{|1-z|}+\frac{1}{|1-z|},
\end{align*}
which can then be written as
\begin{align*}
\Psi(z) = \frac1{1-z} + O\left(|1-z|^2N^3\exp\left(-f\left(N^{1/A}\right)\right)\right).
\end{align*}
\end{proof}

Before proving Proposition \ref{thm2-BRgen}, we prove one more lemma to express the error in Lemma \ref{lem1} in terms of $\Psi(z)$, using also the two functions
$$ I(z) := \sum_n z^n = \frac1{1-z} - 1 = \frac{z}{1-z} $$
and
$$ K_N(z) := \sum_{n\le N} \frac1{z^n}
= \frac1z + \frac1{z^2} + \frac1{z^3} + \cdots + \frac1{z^N}
= \frac1{z^N}\left(\frac{1-z^N}{1-z}\right), $$
where $|z|<1$.
\begin{lem}[Goldston and Suriajaya]\label{lem4-psi2-0}
Recall
\[ \psi^0_2(n) = \sum_{m+m'=n}\Lambda_0(m)\Lambda_0(m'), \]
as in the proof of Lemma \ref{lem1}.
Then for $N\ge4$, we have
\[ \sum_{n\le N} \psi^0_2(n) = \frac1{2\pi i} \int_{|z|=r} \left(\Psi(z) - I(z)\right)^2 K_N(z)\, \dfrac{dz}z. \]
\end{lem}
We remark that this has been shown in \cite[Equation (13) of Theorem 1]{GoldSur-Fujii}. A similar approach has also been used in \cite[Section 3]{Gold-Yang}.

\begin{proof}[Proof of Lemma \ref{lem4-psi2-0}]
Recall the equality \eqref{psi0_2-eq} that
\begin{equation}\label{lem4-start}
\psi^0_2(n)= \psi_2(n) - 2\psi(n-1) + (n-1).
\end{equation}
Using the functions $\Psi(z)$ and $I(z)$, we see that
\begin{align*}
&I(z)^2=\sum_n (n-1)z^n, \\
&I(z)\Psi(z)=\sum_n\left(\sum_{m\le n-1}\Lambda(m)\right)z^n=\sum_n\psi(n-1)z^n,
\end{align*}
and, recalling \eqref{smoothedPsisquare},
$$ \Psi(z)^2 = \sum_n \psi_2(n)z^n. $$
We now truncate and remove the weight $z^n$ in these sums. 
By Cauchy's residue theorem, we have
\begin{equation}\label{int-unit_circle}
\frac1{2\pi i} \int_{|z|=r}z^\ell \,dz =
\begin{cases}
1, & \mbox{if } \ell=-1, \\
0, & \mbox{otherwise}.
\end{cases}
\end{equation}
Applying \eqref{int-unit_circle} and using the kernel $K_N(z)$, we have
\begin{align*}
\frac1{2\pi i}\int_{|z|=r}\Psi(z)^2K_N(z)\,\frac{dz}z
=\sum_n\psi_2(n)\sum_{k=1}^N\left(\frac1{2\pi i}\int_{|z|=r}z^{n-k-1}\,dz\right)
=\sum_{n\le N}\psi_2(n),
\end{align*}
and similarly,
\begin{align*}
&\frac1{2\pi i}\int_{|z|=r}I(z)\Psi(z)K_N(z)\,\frac{dz}z =\sum_n\psi(n-1)\sum_{k=1}^N\left(\frac1{2\pi i}\int_{|z|=r}z^{n-k-1}\,dz\right)
=\sum_{n\le N}\psi(n-1), \\
&\frac1{2\pi i}\int_{|z|=r}I(z)^2K_N(z)\,\frac{dz}z =\sum_n(n-1)\sum_{k=1}^N\left(\frac1{2\pi i}\int_{|z|=r}z^{n-k-1}\,dz\right)
=\sum_{n\le N}(n-1).
\end{align*}
Therefore by \eqref{lem4-start},
\begin{align*}
\sum_{n\le N} \psi^0_2(n)
&= \sum_{n\le N} \left(\psi_2(n) - 2\psi(n-1) + (n-1)\right) \\
&= \frac1{2\pi i} \int_{|z|=r} \left(\Psi(z)^2-2\Psi(z)I(z)+I(z)^2\right) K_N(z) \,\frac{dz}z \\
&= \frac1{2\pi i} \int_{|z|=r} \left(\Psi(z)-I(z)\right)^2 K_N(z) \,\frac{dz}z.
\end{align*}
\end{proof}

We are now ready to prove Proposition \ref{thm2-BRgen}.

\begin{proof}[Proof of Proposition \ref{thm2-BRgen}]
Recall the formula \eqref{lem1'}. By Assumption \ref{Gassumption},
\begin{align*}
R_1(N) &\ll \left|G(N) - \frac{N^2}2\right| + \left|\sum_{n\le N} \psi^0_2(n)\right| + |O(N)| \\
&\ll N^2\exp(-f(N)) + \left|\sum_{n\le N} \psi^0_2(n)\right| + |O(N)|.
\end{align*}
Applying Lemma \ref{lem4-psi2-0} to the second term, we have
\begin{equation*}
\begin{aligned}
\sum_{n\le N} \psi^0_2(n)
&= \frac1{2\pi i} \int_{|z|=r} \left(\Psi(z) - I(z)\right)^2 K_N(z)\, \dfrac{dz}z \\
&= \int_0^1 \left( \Psi(re^{2\pi i\alpha})-I(re^{2\pi i\alpha}) \right)^2 K_N(re^{2\pi i\alpha}) \,d\alpha 
\end{aligned}
\end{equation*}
Thus
\begin{equation}\label{R_1bound-with-arc}
R_1(N) \ll N^2\exp(-f(N)) + \int_0^{1/2} \left|\Psi(re^{2\pi i\alpha})-I(re^{2\pi i\alpha})\right|^2 \left|K_N(re^{2\pi i\alpha})\right| \,d\alpha,
\end{equation}
where we used that the integrand is even with period one. It remains to bound this integral.

We first bound $K_N(z)$ where $z=re^{2\pi i\alpha}$ with $r=e^{-1/N}$.
We easily see that
\[ |K_N(z)|\le \frac{1}{r^N}\cdot \frac{1+r^N}{|1-z|}
\le \frac{1}{r^N}\cdot \frac{2}{|1-z|}
\ll \frac{1}{|1-z|}. \]
On the other hand, we can also bound $K_N(z)$ as
\[ |K_N(z)|
\le \frac{1}{|z|^N} \sum_{j=0}^{N-1}|z|^j \le \frac{N}{\left(e^{-1/N}\right)^N} \ll N. \]
Hence
\begin{equation}\label{thm4bound1}
K_N(z) \ll \min\{ N,\frac{1}{|1-z|}\}.
\end{equation}

Next we estimate $|1-z|$. We note that
\begin{align*}
|1-z|^2 &= |1-re^{2\pi i\alpha}|^2 
= \left(1-r\cos(2\pi\alpha)\right)^2 + r^2\sin^2(2\pi\alpha)
= (1-r)^2+4r\sin^2(\pi\alpha).
\end{align*}
The first term above is
$$ (1-r)^2 = \left(1-e^{-1/N}\right)^2 = \left(\frac{1}{N} + O\left(\frac{1}{N^2}\right)\right)^2 \asymp \frac{1}{N^2}. $$
Meanwhile, for the second term, we make use of the inequality
$$ \frac2\pi x \le \sin{x} \le x \qquad\text{for}\quad 0\le x\le\pi/2 $$
which implies
$$ 4r\sin^2(\pi\alpha) \asymp \alpha^2 \qquad\text{for all}\quad 0\le\alpha\le1/2. $$
Therefore
\begin{equation}\label{thm4bound2}
|1-z|^2\asymp \max\{\frac{1}{N^2},\alpha^2\} \qquad\text{for}\quad 0\le\alpha\le1/2,
\end{equation}
which also implies for $0\le\alpha\le1/2$,
\begin{equation}\label{thm4bound3}
|1-z|\asymp \max\{\frac{1}{N},|\alpha|\}
\qquad\text{and}\qquad
\frac1{|1-z|}\asymp \min\{N,\frac{1}{|\alpha|}\}.
\end{equation}

Finally we bound the integral in \eqref{R_1bound-with-arc}.
Let $1/N\le\delta\le1/4$.
We write
\begin{align*}
&\int_0^{1/2} \left|\Psi(re^{2\pi i\alpha})-I(re^{2\pi i\alpha})\right|^2 \left|K_N(re^{2\pi i\alpha})\right| \,d\alpha \\
&\qquad= \left(\int_0^\delta + \int_{\delta}^{1/2}\right) \left|\Psi(re^{2\pi i\alpha})-I(re^{2\pi i\alpha})\right|^2 \left|K_N(re^{2\pi i\alpha})\right| \,d\alpha.
\end{align*}
On the arc $\alpha\in[0,\delta]$, we note that
\[ I(z) = \frac1{1-z} + O(1), \]
thus we can write
\begin{align*}
&\int_0^{\delta} \left( \Psi(re^{2\pi i\alpha})^2-I(re^{2\pi i\alpha}) \right)^2 \left|K_N(re^{2\pi i\alpha})\right| \,d\alpha \\
&\qquad= \int_0^{\delta} \left( \Psi(re^{2\pi i\alpha})-\frac1{1-re^{2\pi i\alpha}} + O(1) \right)^2 K_N(re^{2\pi i\alpha})\,d\alpha.
\end{align*}
Applying Lemma \ref{lem3-thm2}, we have that the above is
\begin{align*}
\ll \int_0^{\delta} \left(|1-re^{2\pi i\alpha}|^2\,N^3\exp\left(-f\left(N^{1/A}\right)\right) + |O(1)| \right)^2 \left|K_N(re^{2\pi i\alpha})\right| \,d\alpha.
\end{align*}
By \eqref{thm4bound1}, \eqref{thm4bound2} and \eqref{thm4bound3}, this is bounded by
\begin{align*}
&\ll \int_0^{\frac1N} \left(\frac1{N^2}\,N^3\exp\left(-f\left(N^{1/A}\right)\right)\right)^2 N\,d\alpha + \int_\frac1N^{\delta} \left(\alpha^2N^3\exp\left(-f\left(N^{1/A}\right)\right)\right)^2\,\frac{d\alpha}{\alpha} \\
&= N^2\exp\left(-2f\left(N^{1/A}\right)\right) + N^6\exp\left(-2f\left(N^{1/A}\right)\right)\int_\frac1N^{\delta}\alpha^3\,d\alpha \\
&\ll \delta^4N^6\exp\left(-2f\left(N^{1/A}\right)\right).
\end{align*}

On the arc $[\delta,1/2]$, we estimate the integral as
\begin{align*}
\int_\delta^{1/2} \left|\Psi(re^{2\pi i\alpha})-I(re^{2\pi i\alpha})\right|^2 \left|K_N(re^{2\pi i\alpha})\right| \,d\alpha
&\ll \int_\delta^{1/2} \left|\Psi(re^{2\pi i\alpha})-I(re^{2\pi i\alpha})\right|^2 \frac{1}{|\alpha|}d\alpha \\
&\le \frac{1}{\delta} \int_\delta^{1/2} \left| \sum_n (\Lambda(n)-1)r^n e^{2\pi in\alpha} \right|^2 \,d\alpha \\
&\le \frac{1}{\delta} \int_0^{1} \left| \sum_n (\Lambda(n)-1)r^n e^{2\pi i n \alpha}\right|^2 \,d\alpha.
\end{align*}
Using Parseval's identity, we have that the integral on the right-hand side of the last inequality equals $\sum_n (\Lambda(n)-1)^2 r^{2n}$. Hence the integral over the arc $[\delta,1/2]$ is
\begin{align*}
\ll \frac{1}{\delta}\sum_n \Lambda(n)^2 r^{2n} = \frac{1}{\delta} \sum_n \Lambda(n)^2 e^{-2n/N}.
\end{align*}
Next we truncate the sum at $N$ and write
\begin{align*}
\sum_n \Lambda(n)^2 e^{-2n/N} = \sum_{n\le N} \Lambda(n)^2 e^{-2n/N} + \sum_{n> N} \Lambda(n)^2 e^{-2n/N}
=: S_1+S_2.
\end{align*}
For the sum $S_1$, PNT immediately gives
\begin{align*}
S_1 = \sum_{n\le N} \Lambda(n)^2 e^{-2n/N} \le \sum_{n\le N} \Lambda(n)^2 \le \Lambda(N) \sum_{n\le N} \Lambda(n)
= \psi(N)\Lambda(N) \le \psi(N)\log{N}
\ll N \log{N}.
\end{align*}
Thanks to the exponential term, $S_2$ is also bounded by $N \log{N}$: Using, for example, partial summation, we see that 
\begin{align*}
S_2 = \sum_{n>N} \Lambda(n)^2 e^{-2n/N}
&\ll \sum_{n\le N} \Lambda(n)^2 + \frac1N \int_N^\infty \left(\sum_{n\le t} \Lambda(n)^2 \right) e^{-2t/N} \,dt \\
&\ll N\log N + \frac1N \int_N^\infty t(\log{t}) e^{-2t/N} \,dt \\
&\ll N\log N + N\log N \int_2^\infty t(\log{t}) e^{-t} \,dt \\
&\ll N\log N.
\end{align*}
Therefore $S_1+S_2 \ll N \log N$.

Combining the bounds on both arcs, we have
$$
\int_0^{1/2} \left|\Psi(re^{2\pi i\alpha})-I(re^{2\pi i\alpha})\right|^2 \left|K_N(re^{2\pi i\alpha})\right| \,d\alpha
\ll \delta^4N^6\exp\left(-2f\left(N^{1/A}\right)\right) + \frac{N \log N}{\delta}.
$$
Balancing both terms on the right-hand side, we get
$$
\delta = N^{-1} (\log N)^{1/5} \exp\left({\textstyle \frac25} f\left(N^{1/A}\right)\right),
$$
and therefore
$$
\int_0^{1/2} \left|\Psi(re^{2\pi i\alpha})-I(re^{2\pi i\alpha})\right|^2 \left|K_N(re^{2\pi i\alpha})\right| \,d\alpha
\ll N^2\exp\left(-{\textstyle \frac25}f\left(N^{1/A}\right)\right)(\log N)^{4/5}.
$$
Substituting this into \eqref{R_1bound-with-arc}, we obtain
\begin{align*}
R_1(N) 
&\ll N^2\exp(-f(N)) + N^2\exp\left(-{\textstyle \frac25}f\left(N^{1/A}\right)\right)(\log N)^{4/5} \\
&\ll N^2\exp\left(-{\textstyle \frac25}f\left(N^{1/A}\right)\right)(\log N)^{4/5}.
\end{align*}
By \eqref{RfromR1}, the above bound then implies
\begin{equation}\label{Rbound-exact}
R(x) \ll x (\log x)^{2/5} \exp\left(-\frac15 f\left(\left(\frac{x}2\right)^{1/A}\right)\right),
\end{equation}
which means that there exists a constant $0<C<1$ such that
\[ R(x) \ll x \exp\left(-C f\left(\left(\frac{x}2\right)^{1/A}\right)\right). \]
\end{proof}

The proof of Proposition \ref{thm2-BRgen}, together with Lemmas \ref{lem2-thm2} and \ref{lem3-thm2}, follows the proof of Bhowmik and Ruzsa \cite{BhowRuzsa2018} with a few modifications which somewhat simplify the proof. If we use our Assumption \ref{Gassumption} in their proof, we obtain the insignificantly weaker bound
\[ R(x) \ll x(\log{x})^{1/2}\exp\left(-\frac16f\left(x^{1/A}\right)\right). \]
Finally, using Proposition \ref{thm2-BRgen} and \ref{Ingham-b} of \hyperref[thm-InghamPintz]{PNT-ZFR} we easily complete the proof of Theorem \ref{thm2}.

\begin{proof}[Proof of Theorem \ref{thm2}]
Let $\eta(u)$ be a function satisfying all the conditions \hyperref[thm-InghamPintz]{PNT-ZFR} including the additional condition $\lim_{u\to\infty}\eta'(u)=0$ imposed in \ref{Ingham-b}.
For $\varpi(x)$ as defined in \eqref{converse_omega}, we assume that for some constant $C>5A$,
\[ G(N) = \frac{N^2}{2} + O\left(N^2\exp(-C\varpi(N))\right), \qquad \text{for} \quad N\ge4. \]
%and such that for sufficiently large $x$
%$$ \frac15 \left(C \varpi\left(\left(\frac{x}2\right)^{1/A}\right)-2\log\log{x}\right) \ll C' \varpi\left(\left(\frac{x}2\right)^{1/A}\right) $$
%for some $C'>1$.
Taking $f(x) = C\varpi(x)$ in Proposition \ref{thm2-BRgen}, we have for any real $A>1$ and sufficiently large $x$,
\[ R(x) = O\left(x\exp\left(-C'\varpi\left(\left(\frac{x}2\right)^{1/A}\right)\right)\right) \]
for some $C'>A$.
Note that
\begin{align*}
\varpi \left(\left(\frac{x}2\right)^{1/A}\right)
&= \frac1A\min_{u\ge0}\left(\eta(u)\log\left(\frac{x}2\right)+Au\right) \\
&\ge \frac1A\min_{u\ge0}\left(\eta(u)\log\left(\frac{x}2\right)+u\right)
= \frac1A\varpi\left(\frac{x}2\right),
\end{align*}
thus
\[ R(x) = O\left(x\exp\left(-C_0\varpi\!\left(\frac{x}2\right)\right)\right) \]
for some $C_0>1$.
Following the argument in \cite[p. 64]{Ingham1932}, we can remove the $1/2$ factor:
Since $\log{x}-\varpi(x)$ is an increasing function of $x$,
\begin{align*}
\exp\left(-C_0\varpi\!\left(\frac{x}2\right)\right)
&= \exp\left(-C_0\log\frac{x}2 + C_0\log\frac{x}2 - C_0\varpi\!\left(\frac{x}2\right)\right) \\
&\le \exp\left(-C_0\log\frac{x}2 + C_0\log{x} - C_0\varpi(x)\right) \\
&= \exp\left(C_0\log2 - C_0\varpi(x)\right)
= 2^{C_0}\exp\left(-C_0\varpi(x)\right) \\
&\ll \exp(-C_0\varpi(x)).
\end{align*}
Therefore
\[ R(x) = O\left(x\exp(-C_0\varpi(x))\right), \quad\text{for some constant } C_0>1. \]
If $C_0>2$ we can degrade the bound to $R(x) = O\left(x\exp\left(-C_0'\varpi\!(x)\right)\right)$, where $C_0'$ is a constant satisfying $1<C_0'<2$.
By \ref{Ingham-b} of \hyperref[thm-InghamPintz]{PNT-ZFR}, for all sufficiently large $|t|$, $\zeta(\sigma+it)\neq0$ in the region $\sigma > 1 - \eta(\log|t|)$.
\end{proof}

\section{Smooth average --- Proof of Theorem \ref{thm-smooth} and Corollary \ref{cor-smooth1}}
\label{sec:smooth}

In this final section, we consider a smooth average of the number of Goldbach representations, analogous to the function $\Psi(z)^2$ we used in Section \ref{sec-converse}.
We consider the quantity
\[ F(N) = \sum_n \psi_2(n)e^{-n/N}, \]
as defined in Theorem \ref{thm-smooth}.
If we set
\[ \Psi(N) := \sum_n\Lambda(n)e^{-n/N}, \]
which is the $\Psi$-function defined in \eqref{smoothedPsi} with $z=e^{-1/N}$, we can then write
\begin{align*}
\Psi(N)^2 = \sum_m\sum_{m'} \Lambda(m)\Lambda(m')e^{-(m+m')/N}
= \sum_n\psi_2(n)e^{-n/N} = F(N).
\end{align*}
Using the explicit formula
\begin{equation}\label{PsiExplicit}
\Psi(N) = N -\sum_\rho \Gamma(\rho) N^\rho -\log 2\pi + O\left(\frac1N\right), \end{equation}
see for example \cite[12.1.1 Exercise 8 (c)]{MontgomeryVaughan2007},
where, again, the sum is over the zeros $\rho=\beta+i\gamma$ of the Riemann zeta-function satisfying $\gamma\neq0$, we have
\begin{equation}\label{PsiSmoothSumeq}
\begin{aligned}
F(N) = N^2-2\sum_\rho \Gamma(\rho) N^{\rho+1} + (\Psi(N)-N)^2 + O(N).
\end{aligned}
\end{equation}
This is our key formula which also makes the smooth average $F(N)$ easy to deal with.

\begin{proof}[Proof of Theorem \ref{thm-smooth}]
By Stirling's formula (see for example \cite[Theorem C.1 or the first equation on p. 524]{MontgomeryVaughan2007}), we have, for $0\le \sigma \le 1$ and $|t|\ge 1$, 
\begin{equation*}%\label{GammaStirling}
\Gamma(\sigma + it) \ll e^{-|t|},
\end{equation*}
see \cite[Equation (C.19)]{MontgomeryVaughan2007}.
Thus
\begin{equation}\label{smooth-gamma_sum}
\left| \sum_{\rho} {\Gamma\left( \rho \right)N^{\rho}} \right|
\leq \sum_{\rho=\beta+i\gamma} \left| \Gamma\left( \rho \right) \right| \left|N^{\beta + i\gamma}\right| 
\ll \sum_{\rho=\beta+i\gamma} e^{-|\gamma|} N^{\beta}.
\end{equation}

We now assume that
$\zeta(\sigma+it)\neq0$ in some region $\sigma>1-\eta(|t|)$, where $0<\eta(u)\le1/2$.
Recall that non-real zeros of the Riemann zeta-function are symmetric with respect to the real axis, thus truncating the sum at some height $T>0$ we have
\begin{equation}\label{smooth-breaking_sum}
\sum_{\rho=\beta+i\gamma} e^{-|\gamma|} N^{\beta}
\ll \sum_{\gamma>0} e^{-\gamma} N^{1-\eta(\gamma)}
\ll \sum_{0<\gamma\le T} e^{-\gamma} N^{1-\eta(\gamma)}
+ \sum_{\gamma>T} e^{-\gamma} N^{1-\eta(\gamma)}.
\end{equation}
The number of zeros of the Riemann zeta-function on each vertical segment of length $1$ is
$$ \sum_{T<\gamma \le T+1} 1 \ll \log T, $$
see \cite[Theorem 25]{Ingham1932}, \cite[Theorem 9.4]{Titchmarsh}, or \cite[Corollary 14.3]{MontgomeryVaughan2007}.
Applying this bound, we can show that
$$ \sum_{0<\gamma\le T} e^{-\gamma} \ll 1. $$
Hence the first sum on the right-hand side of \eqref{smooth-breaking_sum} is $\ll N^{1-\eta(T)}$.
For the second sum, we bound $N^{1-\eta(\gamma)}$ trivially by $N$ and thus
\begin{align*}
\sum_{\gamma>T} e^{-\gamma} N^{1-\eta(\gamma)}
\ll N \sum_{\gamma>T} e^{-\gamma}
\ll N e^{-T}\log T.
\end{align*}
Substituting the above into \eqref{smooth-gamma_sum} we have
\begin{equation}\label{gamma-sum}
\sum_\rho \Gamma(\rho) N^{\rho} \ll N^{1-\eta(T)} + N e^{-T}\log T.
\end{equation}

As discussed in Section \ref{sec:Goldbach-error}, a zero-free region $1-\eta(u)$ of the Riemann zeta-function satisfies either
$$
\log^{a_1}T \ll \frac1{\eta(T)} \ll \log^{a_2}T \qquad\text{for some constants } 0<a_1<a_2\le1,
$$
which implies
$$ \log\frac1{\eta(T)} \asymp \log\log T, $$
or, the quasi-RH or RH case where
$$ \log\frac1{\eta(T)} \ll 1. $$
Thus balancing the two terms on the right-hand side of \eqref{gamma-sum} leads us to an optimal choice $T=\log N$ and we have
\begin{align*}
\sum_\rho \Gamma(\rho) N^{\rho} \ll N^{1-\eta(\log N)} + \log\log N
\ll N^{1-\eta(\log N)}.
\end{align*}
Therefore by \eqref{PsiExplicit} and \eqref{PsiSmoothSumeq}, we have
\begin{align*}
F(N) &= N^2-2\sum_\rho \Gamma(\rho) N^{\rho+1} + \left(\sum_\rho \Gamma(\rho) N^{\rho}+O(1)\right)^2 + O(N) \\
&= N^2 + O\left(N^{2-\eta(\log N)}\right) + O\left(N^{2(1-\eta(\log N))}\right) \\
&= N^2 + O\left(N^{2 - \eta(\log N)}\right).
\end{align*}
\end{proof}

\begin{proof}[Proof of Corollary \ref{cor-smooth1}]
Assume $\zeta(\sigma+it)\neq0$ in the region
$$ \sigma > 1 - \eta(|t|), \qquad |t|\ge3, $$
where
\begin{equation*}%\label{KorVino-shape-zfr2}
\eta(u) = \frac{c}{(\log{u})^a(\log\log{u})^b}
\end{equation*}
for some $c>0$, $0\le a\le1$, and $b\in\mathbb{R}$ but $b\ge0$ if $a=0$.
Then we have for $N\ge16$,
\[ N^{1-\eta(\log{N})}
=Ne^{-\frac{c\log{N}}{(\log\log{N})^a(\log\log\log{N})^b}}, \]
and by Theorem \ref{thm-smooth},
\[ F(N) = N^2 + O\left(N^{2}\exp\left(-\frac{c\log N}{{\left( \log{\log N} \right)^{a}}{\left( \log{\log{\log N}} \right)^{b}}}\right)\right). \]
\end{proof}

\section*{Acknowledgement and Funding}

This work was supported by MEXT Initiative for Realizing Diversity in the Research Environment through Kyushu University's Diversity and Super Global Training Program for Female and Young Faculty (SENTAN-Q).
The fourth author was also supported by JSPS KAKENHI Grant Numbers 18K13400 and 22K13895.
We thank Professor Daniel A. Goldston for giving us lots of valuable remarks and suggestions on this project.
We also thank the American Institute of Mathematics where some parts of this project were done.


\begin{thebibliography}{ABCD199}

\bibitem[BS-P10]{BhowPuchta2010} G. Bhowmik, J. -C. Schlage-Puchta, {\it Mean representation number of integers as the sum of primes}, Nagoya Math. J., {\bf 200} (2010), 27--33.

\bibitem[BR18]{BhowRuzsa2018} G. Bhowmik, I. Z. Ruzsa, {\it Average Goldbach and the quasi-Riemann hypothesis}, Anal. Math. {\bf 44} (2018), no. 1, 51--56. 

\bibitem[BHMS19]{Bhowmik-H-M-SGoldbach2019} Gautami Bhowmik, Karin Halupczok, Kohji Matsumoto, Yuta Suzuki, {\it Goldbach representations in arithmetic progressions and zeros of Dirichlet L-functions}, Mathematika {\bf 65} (2019), no. 1, 57--97.
 
%\bibitem[BH20]{BhowHalup2020} Gautami Bhowmik and Karin Halupczok, {\it Asymptotics of Goldbach Representations}, Proceedings of Various Aspects of Multiple Zeta Functions — in honor of Professor Kohji Matsumoto’s 60th birthday, Advanced Studies in Pure Mathematics {\bf 84} (2020), 1--21.

% \bibitem[BKP19]{BKP} J\"org Br\"udern, Jerzy Kaczorowski, Alberto Perelli, {\it Explicit formulae for averages of Goldbach representations}, Trans. Amer. Math. Soc. {\bf 372} (2019), no. 10, 6981--6999.

%\bibitem[Cra21]{Cramer21} H. Cram\'er, {\it Some theorems concerning prime numbers}, Arxiv f{\"o}r Mat. Astr. Fys. {\bf 15} (1921), no. 5, 33 pp.

%\bibitem{Davenport} 

%\bibitem[EM07]{EgamiMatsumoto} S. Egami AND K. Matsumoto, {\it Convolutions of the von Mangoldt function and related Dirichlet series}, Proceedings of the 4th China-Japan Seminar held at Shandong, 1--23, S. Kanemitsu and J. -Y. Liu eds., Ser. Number Theory Appl., 2, World Sci. Publ., Hackensack, NJ, 2007.

%\bibitem[EF56]{ErdosFuchs1956} P. Erd\H{o}s and W. H. J. Fuchs, {\it On a problem of additive number theory}, J. London Math. Soc., {\bf 31} (1956), 67--73.

%\bibitem[FG95]{FriedlanderGoldston95} J. B. Friedlander and D. A. Goldston, {\it Some Singular Series Averages and the distribution of Goldbach numbers in short intervals}, Illinois Journal of Mathematics {\bf 39}, no. 1, Spring 1995, 158--180.

\bibitem[Fuj91a]{Fujii1} A. Fujii, {\it An additive problem of prime numbers}, Acta Arith. {\bf 58} (1991), 173--179.

\bibitem[Fuj91b]{Fujii2} A. Fujii, {\it An additive problem of prime numbers. II}, Proc. Japan Acad. ser. A Math. Sci. {\bf 67} (1991), 248--252.

\bibitem[Fuj91c]{Fujii3} A. Fujii, {\it An additive problem of prime numbers. III}, Proc. Japan Acad. ser. A Math. Sci. {\bf 67} (1991), 278--283.

\bibitem[Gol83]{Gold83} D. A. Goldston, {\it On a result of Littlewood concerning prime numbers II}, Acta Arithmetica {\bf XLIII} (1983), 49–-51.

%\bibitem[Gol05]{Gold2005} D. A. Goldston, {\it Notes on pair correlation of zeros and prime numbers}, Recent Perspectives in Random Matrix Theory and Number Theory, LMS Lecture Note Series {\bf 322}, Edited by F. Mezzadri and N. C. Snaith, Cambridge Unversity Press, 2005, 79--110.

%\bibitem[GM87]{GoldMont} D. A. Goldston and H. L. Montgomery, {\it Pair correlation of zeros and primes in short intervals}, Analytic Number Theory and Diophantine Problems (A. C. Adolphson and et al., eds.), Proc. of a Conference at Oklahoma State University (1984), Birkhauser Verlag, 1987, 183--203.

\bibitem[GS22]{GoldSur-PNT-PC} D. A. Goldston and Ade Irma Suriajaya, {\it The Prime Number Theorem and Pair Correlation of Zeros of the Riemann Zeta-Function}, Res. Number Theory {\bf 8} (2022), article number: 71.

\bibitem[GS23a]{GoldSur-Fujii} D. A. Goldston and Ade Irma Suriajaya, {\it On an Average Goldbach Representation Formula of Fujii}, Nagoya Math. J. {\bf 250} (2023), 511--532.

\bibitem[GS23b]{GoldSur-forWirsing} D. A. Goldston and Ade Irma Suriajaya, {\it On a result of Granville concerning Goldbach representations}, Number Theory in Memory of Eduard Wirsing (2023), H. Maier et al. (eds.), Springer Nature Switzerland AG, 145--156.

%\bibitem[GV96]{GV1996} D. A. Goldston and R. C. Vaughan, {\it On the Montgomery-Hooley asymptotic formula}, Sieve Methods, Exponential Sums, and their Application in Number Theory, Greaves, Harman, Huxley Eds., Cambridge University Press, (1996), 117--142.

\bibitem[GY17]{Gold-Yang} D. A. Goldston and L. Yang, {\sl The Average Number of Goldbach Representations}, in: Prime numbers and representation theory, Edited by Ye Tian \& Yangbo Ye, Science Press, Beijing, 2017, 1--12.

%\bibitem[Gon93]{Gonek93} S. M. Gonek, {\it An explicit formula of Landau and its applications to the theory of the zeta-function}, Contemp. Math. {\bf 143} (1993), 395--413.

\bibitem[Gra07]{Gran1} A. Granville, {\it Refinements of Goldbach's conjecture, and the generalized Riemann hypothesis}, Funct. Approx. Comment. Math. {\bf 37} (2007), 159--173.

\bibitem[Gra08]{Gran2} A. Granville, {\it Corrigendum to \lq\lq Refinements of Goldbach's conjecture, and the generalized Riemann hypothesis\rq\rq}, Funct. Approx. Comment. Math. {\bf 38} (2008), 235--237.

%\bibitem[Gro65]{Grosswald} {\'E}mile Grosswald, {\it Sur l'ordre de grandeur des diff{\'e}rences $\psi(x)-x$ et $\pi(x)-\text{li}\, x$}, C. R. Acad. Sci. Paris {\bf 260} (1965), 3813--3816. 

%\bibitem[HL18]{HardyLittlewood1918} G. H. Hardy and J. E. Littlewood, {\it Contributions to the theory of the Riemann zeta-function and the theory of the distribution of primes}, Acta Math. {\bf 41} (1918), 119--196. Reprinted as pp. 20--97 in {\it Collected Papers of G. H. Hardy}, Vol. II, Clarendon Press, Oxford University Press, Oxford, 1967.

%\bibitem[HL19]{HardyLittlewood1919} G. H. Hardy and J. E. Littlewood, {\it Note on Messrs. Shah and Wilson's paper entitled: Ona an empirical formula connected with Goldbach's Theorem}, Proceedings of the Cambridge Philosophical Society, vol. {\bf 19} (1919), 245--254. Reprinted as pp. 535--544 in {\it Collected Papers of G. H. Hardy}, Vol. I, Clarendon Press, Oxford University Press, Oxford, 1966.

%\bibitem[HL22]{HardyLittlewood1922} G. H. Hardy and J. E. Littlewood, {\it Some problems of `Partitio numerorum'; III: On the expression of a number as a sum of primes}, Acta Math. {\bf 44} (1922), no. 1, 1--70. Reprinted as pp. 561--630 in {\it Collected Papers of G. H. Hardy}, Vol. I, Clarendon Press, Oxford University Press, Oxford, 1966.

%\bibitem[Hea82]{Heath-Brown} D. R. Heath-Brown, {\it Gaps between primes, and the pair correlation of zeros of the zeta-function}, Acta Arith. {\bf 41} (1982), 85--99.

\bibitem[Ing32]{Ingham1932} A. E. Ingham, {\it The Distribution of Prime Numbers}, Cambridge Tracts in Mathematics and Mathematical Physics {\bf 30}, Cambridge Univ. Press, Cambridge, 1932.
%Cambridge Mathematical Library, Cambridge University Press, Cambridge, 1990. Reprint of the 1932 original; With a foreword by R. C. Vaughan.

\bibitem[vonKoc01]{vonKoch1901} H. von Koch, {\it Sur la distribution des nombres premiers}, Acta Math. {\bf 24} (1901), 159--182.

%\bibitem[Kor58]{koro0} N. M. Korobov, {\it On zeros of the $\zeta(s)$ function}, Dokl. Akad. Nauk SSSR (N.S.) {\bf 118} (1958), 431--432. (in Russian)

\bibitem[Kor58a]{koro} N. M. Korobov, {\it Weyl's estimates of sums and the distribution of primes}, Dokl. Akad. Nauk SSSR {\bf 123} (1958), 28--31. (in Russian)

\bibitem[Kor58b]{koro2} N. M. Korobov, {\it Estimates of trigonometric sums and their applications}, Uspehi Mat. Nauk {\bf 13} (1958), no. 4 (82), 185--192. (in Russian)

%\bibitem[Kou19]{Kouk2019} Dimitris Koukoulopoulos, {\it The distribution of prime numbers}, Graduate Studies in Mathematics {\bf 203}, American Mathematical Society, Providence, RI, 2019.

%\bibitem[Lan90]{Landau1900} E. Landau, {\it Ueber die zahlentheoretische Funktion $\phi(n)$ und ihre Beziehung zum Goldbachschen Satz}, G{\"o}ttinger Nachrichten, 1900, 177--186. 

\bibitem[Lan08]{Landau1908} E. Landau, {\it Nouvelle d\'emonstration pour la formule de Riemann sur le nombre des nombres premiers inf\'erieurs \`a une limite donn\'ee, et d\'emonstration d'une formule plus g\'en\'erale pour le cas des nombres premiers d’une progression arithm\'etique}, Ann. l’\'Ecole Norm. Sup. (3) {\bf 25}, 399--442; Collected Works, Vol. 4, Essen: Thales Verlag, 1986, pp. 87–130.

%\bibitem[Lan12]{Landau1912} E. Landau, {\it \"Uber die Nullstellen der Zetafunktion}, Math. Ann. {\bf 71} (1912), no. 4, 548--564.

\bibitem[LZ12]{Lang-Zac1} A. Languasco and A. Zaccagnini, {\it The number of Goldbach representations of an integer}, Proc. Amer. Math. Soc. {\bf 140} (2012), 795--804.

%\bibitem[LZ15]{Lang-Zac2015} A. Languasco and A. Zaccagnini, {\it A Ces\`aro average of Goldbach numbers}, Forum Math. {\bf 27} (2015), no. 4, 1945–-1960,

%\bibitem[Lit22]{Littlewood1922} J. E. Littlewood, {\it Researches in the Theory of the Riemann $\zeta$-Function}, Proc. London Math. Soc. (2) {\bf 20} (1922), 22--27.

\bibitem[Mon71]{Montgomery71} Hugh L. Montgomery, {\it Topics in Multiplicative Number Theory}, Lecture Notes in Mathematics, Vol. 227, Springer-Verlag, Berlin-New York, 1971.

%\bibitem[Mon72]{Montgomery72} H.L. Montgomery, The Pair Correlation of Zeros of the Zeta Function, in: Analytic Number Theory, Proc. Sympos. Pure Math., Vol. XXIV, St. Louis Univ., St. Louis, Mo., 1972, 181--193.

\bibitem[Mon94]{Mont94} Hugh L. Montgomery, {\it Ten lectures on the interface between analytic number theory and harmonic analysis}, CBMS Regional Conference Series in Mathematics, 84. Published for the Conference Board of the Mathematical Sciences, Washington, DC; by the American Mathematical Society, Providence, RI, 1994.

%\bibitem[MV73]{MontgomeryVaughan1973}H. L. Montgomery and R. C. Vaughan, {\it Error terms in additive prime number theory}, Quart. J. Math. Oxford (2), {\bf 24} (1973), 207--216.

%\bibitem[MV90]{MontgomeryVaughan1990} H. L. Montgomery and R. C. Vaughan, {\it On the Erd\H{o}s--Fuchs theorem}, A Tribute to Paul Erd\H{o}s, Cambridge Univ. Press, Cambridge (1990), 331--338.

\bibitem[MV07]{MontgomeryVaughan2007} H. L. Montgomery and R. C. Vaughan, {\it Multiplicative Number Theory}, Cambridge Studies in Advanced Mathematics {\bf 97}, Cambridge University Press, Cambridge, 2007.

\bibitem[MT22]{MT22} Michael J. Mossinghoff and Timothy S. Trudgian, {\it Oscillations in the Goldbach conjecture}, J. Th\'eor. Nombres Bordeaux {\bf 34} (2022), no. 1, 295--307.

\bibitem[Pin80]{Pintz1980} J. Pintz, {\it On the remainder term of the prime number formula, II. On a theorem of Ingham,} Acta Arith. {\bf 37}, 209--220.

\bibitem[Pin84]{Pintz1984} J. Pintz, {\it On the remainder term of the prime number formula and zeros of Riemann's zeta-function, } Number Theory (Noordwijkerhourt, 1983) Lecture Notes in math. 1068. Berlin: Springer-Verlag, 186--197.

%\bibitem[PT21]{PlattTrudgian} D. J. Platt and T. S. Trudgian, {\it The error term in the prime number theorem}, Math. Comp. {\bf 90} (2021), 871-881.

%\bibitem[SV77]{SaffariVaughan} B. Saffari and R. C. Vaughan, {\it On the fractional parts of $x/n$ and related sequences II}, Ann. Inst. Fourier (Grenoble) {\bf 27} (1977), no. 2, 1--30.

\bibitem[Sch03]{Sch03} Erhard Schmidt, {\it \"Uber die Anzahl der Primzahlen unter gegebener Grenze}, Math. Ann. {\bf 57} (1903), no. 2, 195--204.

%\bibitem[Sel43]{Selberg} A. Selberg, {\it On the normal density of primes in small intervals, and the difference between consecutive primes}, Arch. Math. Naturvid. {\bf 47} (1943), 87--105.

\bibitem[Sta61]{Sta61} W. Sta\'s, {\it \"Uber die Umkehrung eines Satzes von Ingham}, Acta Arith. {\bf 6} (1961), 435--446.

\bibitem[Tit86]{Titchmarsh} E. C. Titchmarsh, {\sl The Theory of the Riemann Zeta-Function}, 2nd ed., revised by D. R. Heath-Brown, Clarendon (Oxford), 1986.

\bibitem[Tur50]{Turan1950} P. Tur\'an, {\it On the remainder term of the prime-number formula, II}, Acta Math. Acad. Sci. Hungar. {\bf 1}, 155--166. 

\bibitem[Tur84]{TuranBook} Paul Tur\'an, {\it On a new method of analysis and its applications}, with the assistance of G. Hal\'asz and J. Pintz, and a foreword by Vera T. S\'os, %Pure and Applied Mathematics, A Wiley-Interscience Publication, 
John Wiley \& Sons, Inc. (New York), 1984.

\bibitem[Vin58]{vino} I. M. Vinogradov, {\it A new estimate of the function $\zeta(1+it)$}, Izv. Akad. Nauk SSSR. Ser. Mat. {\bf 22} (1958), 161--164. (in Russian)

\end{thebibliography}
\end{document}